\RequirePackage{fix-cm}
\documentclass[smallextended]{svjour3}
\smartqed

\usepackage{graphicx}

\usepackage{amsfonts,amssymb,amsmath}

\newcommand{\Lip}{{\operatorname{Lip}}}
\newcommand{\C}{C}
\newcommand{\Linf}{L^\infty}
\newcommand{\diam}{\operatorname{diam}}
\newcommand{\dist}{\varrho}

\newcommand{\argmax}[1]{\underset{#1}{\operatorname{argmax} \,}}
\newcommand{\argmin}[1]{\underset{#1}{\operatorname{argmin} \,}}
\newcommand{\esssup}[1]{\underset{#1}{\operatorname{ess\,sup} \,}}

 \newcommand{\USet}{\mathbb{U}}
 \newcommand{\VSet}{\mathbb{V}}

 \newcommand{\PSU}{\mathbf{u}}
 \newcommand{\PSV}{\mathbf{v}}

 \newcommand{\PSP}{\mathbf{p}}
 \newcommand{\PSQ}{\mathbf{q}}

 \newcommand{\SP}{P}
 \newcommand{\SQ}{Q}

 \newcommand{\SPO}{P^0_{w_0, h}}
 \newcommand{\SQO}{Q^0_{w_0, h}}

 \newcommand{\fA}{f_{w_0, h}}

 \newcommand{\gammaA}{{\gamma_{w_0, h}}}
 \newcommand{\sigmaA}{{\sigma_{w_0, h}}}

 \newcommand{\valA}{\rho_{w_0, h}}

 \newcommand{\valh}{\widehat{\rho}}

 \newcommand{\vall}{\widehat{\rho}\,}

 \newcommand{\G}{{G_\ast}}

 \newcommand{\GA}{{G_h^0}}

 \newcommand{\Val}{\rho}
 \newcommand{\valQSF}{\rho^{(u)}}
 \newcommand{\valQSS}{\rho^{(v)}}

 \newcommand{\valAQSF}{\rho^{(p)}_{w_0, h}}
 \newcommand{\valAQSS}{\rho^{(q)}_{w_0, h}}

 \newcommand{\UPROC}{U}
 \newcommand{\VPROC}{V}

\begin{document}

\title{Solution to Zero-Sum Differential Game with Fractional Dynamics via Approximations}

\titlerunning{Zero-sum differential game with fractional dynamics}

\author{Mikhail Gomoyunov}

\institute{M.~Gomoyunov \at
              Krasovskii Institute of Mathematics and Mechanics of the Ural Branch of the Russian Academy of Sciences,
              S. Kovalevskaya Str., 16, Ekaterinburg, Russia \\
              Ural Federal University,
              Mira Str., 32, Ekaterinburg, Russia \\
              \email{m.i.gomoyunov@gmail.com}
}

\date{Received: date / Accepted: date}

\maketitle

\begin{abstract}
    The paper deals with a zero-sum differential game in which the dynamical system is described by a fractional differential equation with the Caputo derivative of an order $\alpha \in (0, 1).$
    The goal of the first (second) player is to minimize (maximize) the value of a given quality index.
    The main contribution of the paper is the proof of the fact that this differential game has the value, i.e., the lower and upper game values coincide.
    The proof is based on the appropriate approximation of the game by a zero-sum differential game in which the dynamical system is described by a first order functional differential equation of a retarded type.
    It is shown that the values of the approximating differential games have a limit, and this limit is the value of the original game.
    Moreover, the optimal players' feedback control procedures are proposed that use the optimally controlled approximating system as a guide.

\end{abstract}

\section{Introduction}
\label{sec_Introducion}

The paper is devoted to the development of the theory of zero-sum differential games (see, e.g., \cite{Bardi_Capuzzo-Dolcetta_1997,Basar_Olsder_1999,Cardaliaguet_Quincampoix_Saint-Pierre_2007,Fleming_Soner_2006,Friedman_1971,Isaacs_1965,Krasovskii_Subbotin_1988,Lukoyanov_2011,Pontryagin_1981} and the references therein) to the case when a motion of a dynamical system is described a fractional differential equation.
For the basics of fractional calculus, theory of fractional differential equations and their applications, the reader is referred to \cite{Diethelm_2010,Kilbas_Srivastava_Trujillo_2006,Miller_Ross_1993,Podlubny_1999,Samko_Kilbas_Marichev_1993}.

Despite the fact that a great number of various control problems in fractional order systems are intensively studied nowadays, only a few works deal with differential games in such systems (see \cite{Bannikov_2017,Chikrii_Matychyn_2011,Mamatov_Alimov_2018,Petrov_2018} and the references therein).
Furthermore, in these works, only some special classes of linear pursuit-evasion differential games are investigated.

In the paper, we follow the game-theoretical approach \cite{Krasovskii_Krasovskii_1995,Krasovskii_1985,Krasovskii_Subbotin_1988,Lukoyanov_2011,Osipov_1971,Subbotin_1995,Subbotin_Chentsov_1981} and consider a quite general formulation of a zero-sum differential game in a fractional order system.
We suppose that a motion of the system is described by a non-linear fractional differential equation with the Caputo derivative of an order $\alpha \in (0, 1).$
The game is considered on a finite time interval.
The goal of the first (second) player is to minimize (maximize) the value of a given quality index evaluating the system's motion.
The main contribution of the paper is the proof of the fact that the considered differential game has the value, i.e., the lower and upper values of the game coincide.

Due to non-local structure of fractional order derivatives, fractional differential equations are used for describing dynamical systems with the memory effects of a special kind.
It makes these equations close to functional differential equations (see, e.g., \cite{Bellman_Cooke_1963,Hale_Lunel_1993,Kolamnovskii_Myshkis_1992}).
In particular, the Riemann-Liouville fractional integral of the order $(1 - \alpha)$ of the solution to the considered fractional differential equation is, by the definition, the solution to the corresponding first order functional differential equation of a neutral type.
It allows us to introduce a differential game in this neutral type system and study it instead of the original game.
However, to the best of our knowledge, there are no results that can be applied for investigating the obtained differential game.
Namely, in \cite{Baranovskaya_2015,Gomoyunov_Lukoyanov_2018,Gomoyunov_Lukoyanov_Plaksin_2017,Lukoyanov_Gomoyunov_Plaksin_2017,Lukoyanov_Plaksin_2015,Maksimov_1991,Nikol'skii_1972}, only some special classes of neutral type systems are considered,
and, in \cite{Vasil'ev_1972}, the game is considered in the classes of players' programm (open-loop) strategies.

Nevertheless, following \cite{Gomoyunov_2018}, based on the finite-difference Gr\"{u}nwald-Let\-ni\-kov formulas for calculation of fractional derivatives (see, e.g., \cite[p.~386]{Samko_Kilbas_Marichev_1993}), one can approximate the obtained differential game in the first order neutral type system by a differential game in a first order retarded type system.
Let us note that differential games in dynamical systems described by functional differential equations of a retarded type are quite well studied (see, e.g., \cite{Krasovskii_Subbotin_1988,Lukoyanov_2011,Osipov_1971} and the references therein), especially in comparison with differential games in neutral type systems.
Thus, applying the results of \cite{Lukoyanov_2000,Lukoyanov_2003,Lukoyanov_2011}, we derive that the approximating differential game has the value, and, moreover, this value is achieved in the appropriate classes of players' positional (closed-loop) strategies.

Further, based on the ideas from \cite{Krasovskii_Kotelnikova_2012} (see also \cite{Lukoyanov_Plaksin_2015}), to establish a connection between the original and approximating differential games, we consider the players' feedback control procedures that use the optimally controlled approximating system as a guide (see, e.g., \cite[\S~8.2]{Krasovskii_Subbotin_1988}).
It allows us to prove that the values of the approximating games have a limit, and this limit coincides with the value of the original game.
The key point here is the mutual aiming procedure between the original and approximating systems \cite{Gomoyunov_2018} that provides the desired proximity between the systems' motions.
Moreover, in particular, we obtain that the proposed players' control procedures with a guide guarantee the game value with a given accuracy, and, in this sense, they can be called optimal.

Let us note also that differential games give a natural formalization of control problems under conditions of unknown disturbances (see, e.g., \cite{Krasovskii_Krasovskii_1995,Krasovskii_1985,Krasovskii_Subbotin_1988,Subbotin_Chentsov_1981}).
In some other frameworks, such control problems in fractional order systems are studied, e.g., in \cite{Jajarmi_Hajipour_Mohammadzadeh_Baleanu_2018,Shen_Lam_2014}.

The rest of the paper is organized as follows.
In Sect.~\ref{sec_Notations}, we introduce the notations, recall the definitions of fractional order integrals and derivatives, and give some of their properties.
In Sect.~\ref{sec_DG}, the considered differential game in a fractional order system is described, and, in particular, the notion of the game value is defined.
The corresponding differential game in a first order neutral type system is discussed in Sect.~\ref{sec_DSNT}.
In Sect.~\ref{sec_ADG}, we propose an approximation of this game by a differential game in a first order retarded type system.
In Sect.~\ref{sec_CPG}, the mutual aiming procedure between the original and approximating systems and the optimal players' control procedures with a guide are described, the limit of the values of the approximating differential game is introduced.
In Sect.~\ref{sec_Proof}, we prove that the original differential game has the value.
Concluding remarks are given in Sect.~\ref{sec_conclusion}.

\section{Notations and Definitions}
\label{sec_Notations}

Let $t_0, \vartheta \in \mathbb{R},$ $t_0 < \vartheta,$ and $n \in \mathbb{N}$ be fixed.
Let $\mathbb{R}^n$ be the $n$-dimen\-sional Euclidian space with the scalar product $\langle \cdot, \cdot \rangle$ and the norm $\|\cdot\|.$
By $\Linf([t_0, \vartheta], \mathbb{R}^n),$ we denote the space of essentially bounded (Lebesgue) measurable functions $x: [t_0, \vartheta] \rightarrow \mathbb{R}^n$ with the norm
\begin{equation*}
    \|x(\cdot)\|_\infty = \esssup{t \in [t_0, \vartheta]} \|x(t)\|.
\end{equation*}
Let $\C([t_0, \vartheta], \mathbb{R}^n)$ be the space of continuous functions $x: [t_0, \vartheta] \rightarrow \mathbb{R}^n$ with the uniform norm, which is also denoted by $\|\cdot\|_\infty.$
Let $\Lip^0([t_0, \vartheta], \mathbb{R}^n)$ be the set of functions $x(\cdot) \in \C([t_0, \vartheta], \mathbb{R}^n)$ that are Lipschitz continuous and satisfy the equality $x(t_0) = 0.$
For $L \geq 0,$ we denote by $\Lip_L^0([t_0, \vartheta], \mathbb{R}^n)$ the set of functions $x(\cdot) \in \Lip^0([t_0, \vartheta], \mathbb{R}^n)$ that satisfy the Lipschitz condition with this constant $L.$

Let $\alpha \in (0, 1)$ be fixed.
For a function $x: [t_0, \vartheta] \rightarrow \mathbb{R}^n,$ the Riemann-Liouville (R.-L.) fractional integral of the order $\alpha$ and the R.-L. fractional derivative of the order $\alpha$ are respectively defined by
\begin{equation} \label{RL_integral_derivative}
    \begin{array}{rcl}
        (I^\alpha x) (t) & = &
        \displaystyle \frac{1}{\Gamma(\alpha)} \int_{t_0}^{t} \frac{x(\tau)}{(t - \tau)^{1 - \alpha}} d\tau, \\[1.2em]
        (D^\alpha x) (t) & = &
        \displaystyle \frac{d}{dt} (I^{1 - \alpha} x) (t) =
        \frac{1}{\Gamma(1 - \alpha)} \frac{d}{dt} \int_{t_0}^{t} \frac{x(\tau)}{(t - \tau)^{\alpha}} d\tau, \quad t \in [t_0, \vartheta],
    \end{array}
\end{equation}
where $\Gamma$ is the gamma function.
For the properties of the fractional order integrals and derivatives, the reader is referred to \cite{Diethelm_2010,Kilbas_Srivastava_Trujillo_2006,Miller_Ross_1993,Podlubny_1999,Samko_Kilbas_Marichev_1993}.
In this section, we shortly describe those properties that are used in the paper.
The details can also be found in \cite{Gomoyunov_2017,Gomoyunov_2018}.

Let $I^\alpha (\Linf([t_0, \vartheta], \mathbb{R}^n))$ be the set of functions $x: [t_0, \vartheta] \rightarrow \mathbb{R}^n$ that can be represented by the R.-L. fractional integral of the order $\alpha$ of a function $\varphi(\cdot) \in \Linf([t_0, \vartheta], \mathbb{R}^n),$ i.e., $x(t) = (I^\alpha \varphi)(t),$ $t \in [t_0, \vartheta].$

Let $x(\cdot) \in I^\alpha (\Linf([t_0, \vartheta], \mathbb{R}^n)).$
Then the derivative $(D^\alpha x)(t)$ exists for almost every $t \in [t_0, \vartheta];$
the inclusion $(D^\alpha x)(\cdot) \in \Linf([t_0, \vartheta], \mathbb{R}^n)$ is valid;
and $x(t) = (I^\alpha (D^\alpha x))(t),$ $t \in [t_0, \vartheta].$
Moreover, there exists $H > 0$ such that, for any $x(\cdot) \in I^\alpha (\Linf([t_0, \vartheta], \mathbb{R}^n)),$ the following inequality holds:
\begin{equation} \label{H_infty}
    \|x(t) - x(t^\prime)\| \leq H \|(D^\alpha x)(\cdot)\|_\infty |t - t^\prime|^\alpha, \quad t, t^\prime \in [t_0, \vartheta].
\end{equation}
Further, let us consider the function $y(t) = (I^{1 - \alpha} x)(t),$ $t \in [t_0, \vartheta].$
Then, according to (\ref{RL_integral_derivative}), we have $\dot{y}(t) = (D^\alpha x)(t)$ for almost every $t \in [t_0, \vartheta],$ where we denote $\dot{y}(t) = d y(t) / d t;$
the inclusion $y(\cdot) \in \Lip^0_L([t_0, \vartheta], \mathbb{R}^n)$ is valid with the constant $L = \|(D^\alpha x)(\cdot)\|_\infty;$
and the following representation formula holds:
\begin{equation} \label{representation_formula}
    x(t)
    = (D^{1 - \alpha}y)(t)
    = \frac{1}{\Gamma(\alpha)} \int_{t_0}^{t} \frac{\dot{y}(\tau)}{(t - \tau)^{1 - \alpha}} d\tau, \quad t \in [t_0, \vartheta].
\end{equation}

Finally, for a function $x: [t_0, \vartheta] \rightarrow \mathbb{R}^n,$ the Caputo (C.) fractional derivative of the order $\alpha$ is defined by
\begin{equation} \label{Caputo_derivative}
    ({}^C D^\alpha x) (t)
    = \big(D^\alpha (x(\cdot) - x(t_0)) \big) (t),
    \quad t \in [t_0, \vartheta].
\end{equation}
In particular, if $x(t_0) = 0,$ then the R.-L. and C. fractional derivatives coincide.

\section{Differential Game with Fractional Dynamics}
\label{sec_DG}

\subsection{Fractional Order System}
\label{sec_DS}

We consider a dynamical system which motion is described by the following fractional differential equation with the C. derivative of the order $\alpha:$
\begin{equation} \label{system}
    \begin{array}{c}
        (^C D^\alpha x) (t) = f(t, x(t), u(t), v(t)), \quad t \in [t_0, \vartheta], \\[0.5em]
        x(t) \in \mathbb{R}^n, \quad u(t) \in \USet, \quad v(t) \in \VSet.
    \end{array}
\end{equation}
Here $t$ is the time; $x(t)$ is the value of the state vector at the time $t;$ $u(t)$ and $v(t)$ are respectively the values of the control vectors of the first and second players at the time $t;$ $t_0$ and $\vartheta$ are called the initial and terminal times; the sets $\USet \subset \mathbb{R}^r$ and $\VSet \subset \mathbb{R}^s$ are compact, $r, s \in \mathbb{N}.$
We suppose that the function $f: [t_0, \vartheta] \times \mathbb{R}^n \times \USet \times \VSet \rightarrow \mathbb{R}^n$ satisfies the following conditions:
\begin{description}
  \item[($A.1$)] The function $f$ is continuous.

  \item[($A.2$)] For any $R \geq 0,$ there exists $\lambda > 0$ such that
    \begin{equation*}
        \|f(t, x, u, v) - f(t, x^\prime, u, v)\| \leq \lambda \|x - x^\prime\|
    \end{equation*}
    for any $t \in [t_0, \vartheta],$ $x, x^\prime \in B(R) = \{y \in \mathbb{R}^n: \, \|y\| \leq R\},$ $u \in \USet,$ and $v \in \VSet.$

  \item[($A.3$)] There exists $c > 0$ such that
    \begin{equation*}
        \|f(t, x, u, v)\| \leq (1 + \|x\|) c
    \end{equation*}
    for any $t \in [t_0, \vartheta],$ $x \in \mathbb{R}^n,$ $u \in \USet,$ and $v \in \VSet.$

  \item[($A.4$)] The saddle point condition in a small game \cite[p.~8]{Krasovskii_Subbotin_1988} or, in another terminology, the Isaacs' condition \cite[p.~35]{Isaacs_1965}, holds, i.e.,
    \begin{equation*}
        \min_{u \in \USet} \max_{v \in \VSet} \langle s, f(t, x, u, v) \rangle
        = \max_{v \in \VSet} \min_{u \in \USet} \langle s, f(t, x, u, v) \rangle
    \end{equation*}
    for any $t \in [t_0, \vartheta]$ and $x, s \in \mathbb{R}^n.$
\end{description}
Note that these conditions are quite typical for the differential games theory with first order dynamics (see, e.g., \cite[p.~7]{Krasovskii_Subbotin_1988}).

\subsection{Admissible Positions of the System}
\label{subsec_Positions}

By a position of system (\ref{system}), we mean a pair $(t, w(\cdot))$ consisting of a time $t \in [t_0, \vartheta]$ and a function $w(\cdot) \in \C([t_0, t], \mathbb{R}^n),$ which is treated as a motion history on the interval $[t_0, t].$
The set of the positions $(t, w(\cdot))$ is denoted by $G.$
A position $(t, w(\cdot)) \in G$ is called admissible if the relations below are valid:
\begin{equation} \label{G}
    \begin{array}{l}
        w(t_0) \in B(R_0), \\[0.5em]
        w(\cdot) \in \{w(t_0)\} + I^\alpha (\Linf([t_0, t], \mathbb{R}^n)), \\[0.5em]
        \|(^C D^\alpha w) (\tau)\| \leq (1 + \|w(\tau)\|) c \text{ for a.e. } \tau \in [t_0, t],
    \end{array}
\end{equation}
where $R_0 > 0$ is a fixed constant, $c$ is the constant from condition $(A.3).$
According to the definition given in Sect.~\ref{sec_Notations}, the second inclusion in (\ref{G}) means that there exists a function $\varphi(\cdot) \in \Linf([t_0, t], \mathbb{R}^n)$ such that $w(\tau) = w(t_0) + (I^\alpha \varphi)(\tau),$ $\tau \in [t_0, t].$
The set of the admissible positions is denoted by $\G.$

\begin{proposition} \label{prop_G_properties}
    The set $\G$ is not empty, and there exist $R_1 > 0,$ $M_1 > 0,$ and $H_1 > 0$ such that, for any $(t, w(\cdot)) \in \G,$ the inequalities below are valid:
    \begin{equation*} \label{prop_G_properties_main}
        \begin{array}{l}
            \|w(\tau)\| \leq R_1, \quad \tau \in [t_0, t], \\[0.5em]
            \|(^C D^\alpha w) (\tau) \| \leq M_1 \text{ for a.e. } \tau \in [t_0, t], \\[0.5em]
            \|w(\tau) - w(\tau^\prime)\| \leq H_1 |\tau - \tau^\prime|^\alpha, \quad \tau, \tau^\prime \in [t_0, t].
        \end{array}
    \end{equation*}
\end{proposition}
\begin{proof}
    Let $t \in [t_0, \vartheta]$ and $w_0 \in B(R_0).$ 
    Let us consider the function $w(\tau) = w_0,$ $\tau \in [t_0, t].$
    According to (\ref{RL_integral_derivative}) and (\ref{Caputo_derivative}), we have $(^C D^\alpha w) (\tau) = 0,$ $\tau \in [t_0, t].$ 
    Hence, the inclusion $(t, w(\cdot)) \in \G$ is valid, and, therefore, the set $\G$ is not empty.

    Further, let us define
    \begin{equation*}
        R_1 = (1 + R_0) E_\alpha ((\vartheta - t_0)^\alpha c) - 1, \quad
        M_1 = (1 + R_1) c, \quad
        H_1 = H M_1,
    \end{equation*}
    where $c$ is the constant from $(A.3),$ $E_\alpha$ is the Mittag-Leffler function (see, e.g., \cite[(1.90)]{Samko_Kilbas_Marichev_1993}), and $H$ is the constant from (\ref{H_infty}).
    Let $(t, w(\cdot)) \in \G.$
    Then, due to (\ref{G}) and the results given in Sect.~\ref{sec_Notations}, we have
    \begin{equation*}
        \|w(\tau) - w(t_0)\|
        = \bigg\| \frac{1}{\Gamma(\alpha)} \int_{t_0}^{\tau} \frac{(^C D^\alpha w)(\xi)}{(\tau - \xi)^{1 - \alpha}} d\xi \bigg\|
        \leq \frac{1}{\Gamma(\alpha)} \int_{t_0}^{\tau} \frac{(1 + \|w(\xi)\|) c}{(\tau - \xi)^{1 - \alpha}} d\xi
    \end{equation*}
    for any $\tau \in [t_0, t],$ and, therefore,
    \begin{equation*}
        \|w(\tau)\|
        \leq R_0 + \frac{c}{\Gamma(\alpha)} \int_{t_0}^{\tau} \frac{1 + \|w(\xi)\|}{(\tau - \xi)^{1 - \alpha}} d\xi,
        \quad \tau \in [t_0, t].
    \end{equation*}
    From this inequality, applying the fractional version of Bellman-Gronwall lemma (see, e.g., \cite[Lemma~6.19]{Diethelm_2010} and also \cite[Lemma~1.1]{Gomoyunov_2017}), we conclude $\|w(\tau)\| \leq R_1,$ $\tau \in [t_0, t].$
    Thus, according to (\ref{G}), we have
    \begin{equation*}
        \|(^C D^\alpha w)(\tau)\|
        \leq (1 + \|w(\tau)\|) c
        \leq (1 + R_1) c = M_1 \text{ for a.e. } \tau \in [t_0, t].
    \end{equation*}
    Finally, by the choice of $H,$ we derive
    \begin{equation*}
        \|w(\tau) - w(\tau^\prime)\|
        \leq H M_1 |\tau - \tau^\prime|^\alpha
        = H_1 |\tau - \tau^\prime|^\alpha, \quad \tau, \tau^\prime \in [t_0, t].
    \end{equation*}
    The proposition is proved. \hfill $\square$
\end{proof}

Let $(t_\ast, w_\ast(\cdot)) \in \G$ and $t^\ast \in [t_\ast, \vartheta].$
By admissible control realizations (controls) of the first and second players on the interval $[t_\ast, t^\ast),$ we mean measurable functions $u: [t_\ast, t^\ast) \rightarrow \USet$ and $v: [t_\ast, t^\ast) \rightarrow \VSet,$ respectively.
The sets of the admissible control realizations of the players are denoted by $\mathcal{U}(t_\ast, t^\ast)$ and $\mathcal{V}(t_\ast, t^\ast).$
Following \cite{Idczak_Kamocki_2011} (see also \cite{Gomoyunov_2017}), by a motion of system (\ref{system}) generated from the initial position $(t_\ast, w_\ast(\cdot))$ by players' control realizations $u(\cdot) \in \mathcal{U}(t_\ast, t^\ast)$ and $v(\cdot) \in \mathcal{V}(t_\ast, t^\ast),$ we mean a function $x(\cdot) \in \{w_\ast(t_0)\} + I^\alpha (\Linf([t_0, t^\ast], \mathbb{R}^n))$ that satisfies the initial condition
\begin{equation} \label{initial_condition}
    x(t) = w_\ast(t), \quad t \in [t_0, t_\ast],
\end{equation}
and, together with $u(\cdot)$ and $v(\cdot),$ satisfies Eq. (\ref{system}) for almost every $t \in [t_\ast, t^\ast].$
For such a motion $x(\cdot)$ and a time $t \in [t_0, t^\ast],$ we denote by $(t, x_t(\cdot))$ the corresponding position of system (\ref{system}), i.e.,
\begin{equation} \label{x_t}
    x_t(\tau) = x(\tau), \quad \tau \in [t_0, t].
\end{equation}

\begin{proposition} \label{prop_existence}
    Let $(t_\ast, w_\ast(\cdot)) \in \G$ and $t^\ast \in [t_\ast, \vartheta].$
    Then any players' control realizations $u(\cdot) \in \mathcal{U}(t_\ast, t^\ast)$ and $v(\cdot) \in \mathcal{V}(t_\ast, t^\ast)$ generate from the initial position $(t_\ast, w_\ast(\cdot))$ a unique motion $x(\cdot)$ of system $(\ref{system}).$
    Moreover, for any $t \in [t_0, t^\ast],$ the inclusion $(t, x_t(\cdot)) \in \G$ is valid.
\end{proposition}
\begin{proof}
    Let $(t_\ast, w_\ast(\cdot)) \in \G,$ $t^\ast \in [t_\ast, \vartheta],$ $u(\cdot) \in \mathcal{U}(t_\ast, t^\ast),$ and $v(\cdot) \in \mathcal{V}(t_\ast, t^\ast).$
    The existence and uniqueness of the corresponding motion $x(\cdot)$ of system (\ref{system}) can be proved by the standard scheme (see, e.g., \cite[Theorem~6.1]{Diethelm_2010}, \cite[Theorem~3.1]{Wang_Zhou_2011}, and also \cite[Theorem~2.1]{Gomoyunov_2017}), if we note that $x(\cdot)$ is the motion of system (\ref{system}) if and only if $x(\cdot)$ satisfies the inclusion $x(\cdot) \in \C([t_0, t^\ast], \mathbb{R}^n),$ initial condition (\ref{initial_condition}), and the integral equation
    \begin{equation*}
        \begin{array}{l}
            \displaystyle
            x(t) = w_\ast(t_0)
            + \frac{1}{\Gamma(\alpha)} \int_{t_0}^{t_\ast} \frac{(^C D^\alpha w_\ast)(\tau)}{(t - \tau)^{1 - \alpha}} d \tau \\[1.2em]
            \displaystyle
            + \frac{1}{\Gamma(\alpha)} \int_{t_\ast}^{t} \frac{f(\tau, x(\tau), u(\tau), v(\tau))}{(t - \tau)^{1 - \alpha}} d \tau,
            \quad t \in [t_\ast, t^\ast].
        \end{array}
    \end{equation*}
    Further, for $t \in [t_0, t_\ast],$ the inclusion $(t, x_t(\cdot)) \in \G$ follows from initial condition (\ref{initial_condition}) and the inclusion $(t_\ast, w_\ast(\cdot)) \in \G.$
    For $t \in (t_\ast, t^\ast],$ the inclusion $(t, x_t(\cdot)) \in \G$ is valid due to $(A.3).$
    The proposition is proved.
    \hfill $\square$
\end{proof}

From Propositions~\ref{prop_G_properties} and~\ref{prop_existence} we derive the following result.
\begin{corollary} \label{cor_motion_properties}
    Let $(t_\ast, w_\ast(\cdot)) \in \G$ and $t^\ast \in [t_\ast, \vartheta].$
    Let $x(\cdot)$ be the motion of system $(\ref{system})$ generated from the initial position $(t_\ast, w_\ast(\cdot))$ by players' control realizations $u(\cdot) \in \mathcal{U}(t_\ast, t^\ast)$ and $v(\cdot) \in \mathcal{V}(t_\ast, t^\ast).$
    Then the following inequalities hold:
    \begin{equation*}
        \|x(t)\| \leq R_1, \quad
        \|x(t) - x(t^\prime)\| \leq H_1 |t - t^\prime|^\alpha, \quad t, t^\prime \in [t_0, t^\ast],
    \end{equation*}
    where the constants $R_1$ and $H_1$ are taken from Proposition~$\ref{prop_G_properties}.$
\end{corollary}

Let us note also the following property of motions of system (\ref{system}), which follows directly from Proposition~\ref{prop_existence}.
Let $(t_\ast, w_\ast(\cdot)) \in \G,$ $t^\ast \in [t_\ast, \vartheta],$ and let $x(\cdot)$ be the motion generated from $(t_\ast, w_\ast(\cdot))$ by $u(\cdot) \in \mathcal{U}(t_\ast, t^\ast)$ and $v(\cdot) \in \mathcal{V}(t_\ast, t^\ast).$
Further, let $t^{\ast \ast} \in [t^\ast, \vartheta],$ and let $x^\ast(\cdot)$ be the motion generated from $(t^\ast, x_{t^\ast}(\cdot))$ by $u^\ast(\cdot) \in \mathcal{U}(t^\ast, t^{\ast \ast})$ and $v^\ast(\cdot) \in \mathcal{V}(t^\ast, t^{\ast \ast}).$
Then $x^\ast(\cdot)$ can be considered as the motion generated from $(t_\ast, w_\ast(\cdot))$ by the realizations
\begin{equation*}
    u^{\ast \ast}(t) =
    \begin{cases}
        u(t), & t \in [t_\ast, t^\ast), \\[0.3em]
        u^\ast(t), & t \in [t^\ast, t^{\ast \ast}),
    \end{cases}
    \quad v^{\ast \ast}(t) =
    \begin{cases}
        v(t), & t \in [t_\ast, t^\ast), \\[0.3em]
        v^\ast(t), & t \in [t^\ast, t^{\ast \ast}).
    \end{cases}
\end{equation*}
In particular, this property allows us to consider step-by-step feedback control procedures for constructing players' control realizations (see Sect.~\ref{sec_CPG}).

\subsection{Quality Index}
\label{subsec_QI}

Let $x(\cdot)$ be the motion of system (\ref{system}) generated from an initial position $(t_\ast, w_\ast(\cdot)) \in \G$ by players' control realizations $u(\cdot) \in \mathcal{U}(t_\ast, \vartheta)$ and $v(\cdot) \in \mathcal{V}(t_\ast, \vartheta).$
Let quality of this motion be evaluated by the index
\begin{equation} \label{quality_index}
    \gamma = \sigma(x(\cdot)).
\end{equation}
We suppose that the function $\sigma: \C([t_0, \vartheta], \mathbb{R}^n) \rightarrow \mathbb{R}$ satisfies the following condition:
\begin{description}
  \item[($A.5$)] The function $\sigma$ is continuous.
\end{description}

For dynamical system (\ref{system}) and quality index (\ref{quality_index}), we consider a zero-sum differential game in which the first player aims to minimize the value of the quality index, and the second player aims to maximize it.

\subsection{Non-anticipative Strategies and the Game Value}
\label{subsec_NAS}

To define the value of the differential game (\ref{system}), (\ref{quality_index}), we consider non-anticipative strategies of the players (see, e.g., \cite[Ch.~VIII]{Bardi_Capuzzo-Dolcetta_1997} and the references therein) and introduce the lower and upper values of the game.
Note that, in another terminology, such strategies are called quasi-strategies (see, e.g., \cite[p.~24]{Subbotin_Chentsov_1981}) or progressive strategies (see, e.g., \cite[\S~XI.4]{Fleming_Soner_2006}).

Let $(t_\ast, w_\ast(\cdot)) \in \G$ be an initial position.
By a non-anticipative strategy of the first player, we mean a function $\alpha: \mathcal{V}(t_\ast, \vartheta) \rightarrow \mathcal{U}(t_\ast, \vartheta)$ with the following property.
For any $t^\ast \in [t_\ast, \vartheta]$ and any second player's control realizations $v(\cdot), v^\prime (\cdot) \in \mathcal{V}(t_\ast, \vartheta),$ if the equality $v(t) = v^\prime (t)$ is valid for almost every $t \in [t_\ast, t^\ast],$ then the corresponding images $u(\cdot) = \alpha(v(\cdot))$ and $u^\prime (\cdot) = \alpha(v^\prime(\cdot))$ satisfy the equality $u(t) = u^\prime (t)$ for almost every $t \in [t_\ast, t^\ast].$
The lower value of the differential game (\ref{system}), (\ref{quality_index}) is defined by
\begin{equation} \label{lower_value}
    \valQSF (t_\ast, w_\ast(\cdot)) = \inf_{\alpha} \ \sup_{v(\cdot) \in \mathcal{V}(t_\ast, \vartheta)} \ \gamma,
\end{equation}
where $\gamma$ is the value of quality index (\ref{quality_index}) that corresponds to the motion $x(\cdot)$ generated from $(t_\ast, w_\ast(\cdot)) \in \G$ by the second player's control realization $v(\cdot)$ and the first player's control realization $u(\cdot) = \alpha(v(\cdot)).$

Similarly, a function $\beta: \mathcal{U}(t_\ast, \vartheta) \rightarrow \mathcal{V}(t_\ast, \vartheta)$ is a non-anticipative strategy of the second player if, for any $t^\ast \in [t_\ast, \vartheta]$ and any $u(\cdot), u^\prime (\cdot) \in \mathcal{U}(t_\ast, \vartheta)$ such that $u(t) = u^\prime (t)$ for almost every $t \in [t_\ast, t^\ast],$ we have $v(t) = v^\prime (t)$ for almost every $t \in [t_\ast, t^\ast],$ where $v(\cdot) = \beta(u(\cdot))$ and $v^\prime (\cdot) = \beta(u^\prime(\cdot)).$
The upper value of the game is defined by
\begin{equation*}
    \valQSS (t_\ast, w_\ast(\cdot))
    = \sup_{\beta} \ \inf_{u(\cdot) \in \mathcal{U}(t_\ast, \vartheta)} \ \gamma.
\end{equation*}

If the lower and upper game values coincide for any initial position $(t_\ast, w_\ast(\cdot)) \in \G,$ then we say that the game has the value
\begin{equation*}
    \Val (t_\ast, w_\ast(\cdot)) = \valQSF (t_\ast, w_\ast(\cdot)) = \valQSS (t_\ast, w_\ast(\cdot)), \quad (t_\ast, w_\ast(\cdot)) \in \G.
\end{equation*}

The goal of the paper is to prove that the differential game (\ref{system}), (\ref{quality_index}) has the value, and, for any initial position $(t_\ast,w_\ast(\cdot)) \in \G,$ construct the players' feedback control procedures that guarantee the game value $\Val (t_\ast, w_\ast(\cdot))$ with a given accuracy $\zeta > 0.$
These results are formulated in Theorem~\ref{thm_Existence} (see Sect.~\ref{sec_Proof}).
The proof of this theorem follows the scheme from \cite[Theorem~2]{Lukoyanov_Plaksin_2015} and is based on the appropriate approximation of the differential game (\ref{system}), (\ref{quality_index}).
Before describing this approximation, in the next section, we rewrite the considered differential game in another form.

\section{Differential Game in a Neutral Type System}
\label{sec_DSNT}

Let $x(\cdot)$ be the motion of system (\ref{system}) generated from an initial position $(t_\ast, w_\ast(\cdot)) \in \G$ by players' control realizations $u(\cdot) \in \mathcal{U}(t_\ast, \vartheta)$ and $v(\cdot) \in \mathcal{V}(t_\ast, \vartheta).$
Let us consider the function
\begin{equation} \label{y}
    y(t) = \big(I^{1 - \alpha} (x(\cdot) - w_\ast(t_0)) \big) (t), \quad t \in [t_0, \vartheta].
\end{equation}
Since $x(\cdot) \in \{w_\ast(t_0)\} + I^\alpha(\Linf([t_0, \vartheta], \mathbb{R}^n)),$ then, according to the results given in Sect.~\ref{sec_Notations}, we have
\begin{equation} \label{y_Lip^0}
    \begin{array}{l}
        y(\cdot) \in \Lip^0([t_0, \vartheta], \mathbb{R}^n), \\[0.5em]
        \dot{y}(t) = (^C D^\alpha x)(t) \text{ for a.e. } t \in [t_0, \vartheta], \\[0.5em]
        x(t) = w_\ast(t_0) + (D^{1 - \alpha} y)(t), \quad t \in [t_0, \vartheta].
    \end{array}
\end{equation}
Substituting these equalities into Eq. (\ref{system}), we obtain that, instead of the original differential game (\ref{system}), (\ref{quality_index}), one can consider the differential game for the dynamical system
\begin{equation} \label{system_neutral_type}
    \dot{y}(t) = f\big(t, w_\ast(t_0) + (D^{1 - \alpha} y)(t), u(t), v(t)\big), \quad t \in [t_\ast, \vartheta],
\end{equation}
under the initial condition
\begin{equation} \label{initial_condition_neutral_type}
    y(t) = \big(I^{1 - \alpha} (w_\ast(\cdot) - w_\ast(t_0))\big)(t), \quad t \in [t_0, t_\ast],
\end{equation}
and the quality index
\begin{equation} \label{quality_index_neutral_type}
    \gamma = \sigma\big( w_\ast(t_0) + (D^{1 - \alpha} y)(\cdot) \big).
\end{equation}
Furthermore, due to (\ref{representation_formula}), one can rewrite Eq. (\ref{system_neutral_type}) as follows:
\begin{equation} \label{system_neutral_type_rewritten}
    \dot{y}(t) = f\Big(t, w_\ast(t_0) + \frac{1}{\Gamma(\alpha)} \int_{t_0}^{t} \frac{\dot{y}(\tau)}{(t - \tau)^{1 - \alpha}} d\tau, u(t), v(t)\Big),
    \quad t \in [t_\ast, \vartheta].
\end{equation}
Note that the right-hand side of Eq. (\ref{system_neutral_type_rewritten}) depends explicitly on the history of the derivative $\dot{y}(\tau)$ for $\tau \in [t_0, t].$
Therefore, in the terminology of the theory of functional differential equations (see, e.g., \cite{Bellman_Cooke_1963,Hale_Lunel_1993,Kolamnovskii_Myshkis_1992}), Eq. (\ref{system_neutral_type_rewritten}) is a functional differential equation of a neutral type.
To the best of our knowledge, in the theory of differential games in neutral type systems (see the references in Introduction), there are no results that can be directly applied for studying the game (\ref{system_neutral_type}), (\ref{quality_index_neutral_type}), and, therefore, the original game (\ref{system}), (\ref{quality_index}) too.
However, as it is shown in the next section, the game (\ref{system_neutral_type}), (\ref{quality_index_neutral_type}) can be approximated by a differential game in a retarded type system.

\section{Approximating Differential Game}
\label{sec_ADG}

Following \cite[Sect.~6]{Gomoyunov_2018}, let us approximate in relations (\ref{system_neutral_type}), (\ref{quality_index_neutral_type}) the fractional derivative $(D^{1 - \alpha} y)(t)$ by the divided fractional difference $h^{\alpha - 1} (\Delta_h^{1 - \alpha} y) (t)$ with a step size $h > 0,$ where (see, e.g., \cite[p.~385]{Samko_Kilbas_Marichev_1993})
\begin{equation} \label{fractional_difference}
    (\Delta_h^{1 - \alpha} y) (t) = \sum_{i = 0}^{[(t - t_0) / h]} (-1)^i \binom{1 - \alpha}{i} y(t - i h), \quad t \in [t_0, \vartheta],
\end{equation}
the symbol $[\tau]$ means the integer part of $\tau \geq 0,$ and $\binom{1 - \alpha}{i}$ are the binomial coefficients.
In this section, we study the differential game obtained after this approximation.

\subsection{Approximating Dynamical System and Quality Index}
\label{subsec_ADS}

Let us fix a vector $w_0 \in B(R_0)$ and a sufficiently small value of the parameter $h > 0.$
Note that, in what follows, the vector $w_0$ corresponds to an initial position $(t_\ast, w_\ast(\cdot)) \in \G$ of system (\ref{system}) such that $w_0 = w_\ast(t_0).$
Taking into account the above, we consider the following zero-sum differential game, determined by these two parameters $w_0$ and $h.$
We introduce the approximating dynamical system which motion is described by the differential equation
\begin{equation} \label{system_y}
    \begin{array}{c}
        \dot{y}(t) = f\big(t, w_0 + h^{\alpha - 1} (\Delta_h^{1 - \alpha} y) (t), p(t), q(t)\big), \quad t \in [t_0, \vartheta], \\[0.5em]
        y(t) \in \mathbb{R}^n, \quad p(t) \in \USet, \quad q(t) \in \VSet,
    \end{array}
\end{equation}
and the approximating quality index
\begin{equation}\label{quality_index_y}
    \gammaA = \sigma \big(w_0 + h^{\alpha - 1} (\Delta_h^{1 - \alpha} y)(\cdot)\big).
\end{equation}
Here $y(t)$ is the value of the state vector; $p(t)$ and $q(t)$ are respectively the values of the control vectors of the first and second players.
The first player minimizes the value of quality index (\ref{quality_index_y}), the second player maximizes it.

Note that, according to (\ref{fractional_difference}), at a time $t \in [t_0, \vartheta],$ the right-hand side of Eq. (\ref{system_y}) depends on the values $y(t - ih)$ for $i \in \overline{0, [(t - t_0) / h]}$ and, in contrast to (\ref{system_neutral_type_rewritten}), does not depend explicitly on the history of the derivative $\dot{y}(\tau),$ $\tau \in [t_0, t].$
Thus, Eq. (\ref{system_y}) is a functional differential equation of a retarded type.
In what follows, dealing with the game (\ref{system_y}), (\ref{quality_index_y}), we mainly use the constructions and results from \cite{Lukoyanov_2000,Lukoyanov_2003,Lukoyanov_2011}.

\begin{remark}
    Let us note that, even in a simple case when original quality index (\ref{quality_index}) is terminal, i.e., $\gamma = \mu(x(\vartheta))$ for a function $\mu: \mathbb{R}^n \rightarrow \mathbb{R},$ the corresponding approximating quality index $\gammaA = \mu(w_0 + h^{\alpha - 1} (\Delta_h^{1 - \alpha} y)(\vartheta))$ is still non-terminal, since, according to (\ref{fractional_difference}), it depends on the values $y(\vartheta - ih)$ for $i \in \overline{0, [(\vartheta - t_0) / h]}.$
\end{remark}

Taking into account (\ref{y}) and (\ref{y_Lip^0}), by a position of approximating system (\ref{system_y}), we mean a pair $(t, r(\cdot)) \in G$ such that $r(t_0) = 0.$
The set of such positions is denoted by $G^0.$
This set is considered with the metric (see, e.g., \cite{Lukoyanov_2003} and also \cite[p.~25]{Lukoyanov_2011})
\begin{equation*}
    \dist\big((t, r(\cdot)), (t^\prime, r^\prime(\cdot))\big)
    = \max \Big\{ \dist^\ast\big((t, r(\cdot)), (t^\prime, r^\prime(\cdot))\big),
    \dist^\ast\big((t^\prime, r^\prime(\cdot)), (t, r(\cdot))\big)\Big\},
\end{equation*}
where $(t, r(\cdot)), (t^\prime, r^\prime(\cdot)) \in G^0,$ and
\begin{equation*}
    \dist^\ast\big((t, r(\cdot)), (t^\prime, r^\prime(\cdot))\big)
    = \max_{\tau \in [t_0, t]} \ \min_{\tau^\prime \in [t_0, t^\prime]} \ \Big((\tau - \tau^\prime)^2 + \|r(\tau) - r^\prime(\tau^\prime)\|^2\Big)^{1/2}.
\end{equation*}
By the right-hand side of Eqs. (\ref{system_y}), (\ref{quality_index_y}), let us define the functions
\begin{equation*}
    \begin{array}{l}
        \fA(t, r(\cdot), p, q) = f\big(t, w_0 + h^{\alpha - 1} (\Delta_h^{1 - \alpha} r) (t), p, q\big), \\[0.5em]
        \sigmaA (y(\cdot)) = \sigma \big(w_0 + h^{\alpha - 1} (\Delta_h^{1 - \alpha} y)(\cdot)\big),
    \end{array}
\end{equation*}
where $(t, r(\cdot)) \in G^0,$ $p \in \USet,$ $q \in \VSet,$ and $(\vartheta,y(\cdot)) \in G^0.$

Directly from properties ($A.1$)--($A.5$) of the functions $f$ and $\sigma$ it follows that these functions $\fA$ and $\sigmaA$ satisfy the following conditions:
\begin{description} 
    \item{($B.1$)} For any $h > 0,$ the functions $\fA$ and $\sigmaA$ are continuous uniformly in $w_0 \in B(R_0).$

    \item{($B.2$)} For any $h > 0$ and any $R \geq 0,$ there exists $\lambda_h > 0$ such that, for any $w_0 \in B(R_0),$ the inequality
        \begin{equation*}
            \|\fA(t, r(\cdot), p, q) - \fA(t, r^\prime(\cdot), p, q)\| \leq \lambda_h \max_{\tau \in [t_0, t]} \|r(\tau) - r^\prime(\tau)\|
        \end{equation*}
        is valid for any $(t, r(\cdot)),$ $(t, r^\prime(\cdot)) \in G^0$ satisfying $\|r(\cdot)\|_\infty \leq R,$ $\|r^\prime(\cdot)\|_\infty \leq R$ and any $p \in \USet,$ $q \in \VSet.$

    \item{($B.3$)} For any $h > 0,$ there exists $c_h > 0$ such that, for any $w_0 \in B(R_0),$ the estimate
        \begin{equation*}
            \|\fA(t, r(\cdot), p, q)\| \leq (1 + \max_{\tau \in [t_0, t]} \|r(\tau)\|) c_h
        \end{equation*}
        holds for any $(t, r(\cdot)) \in G^0,$ $p \in \USet,$ and $q \in \VSet.$

    \item{($B.4$)} For any $w_0 \in B(R_0)$ and any $h > 0,$ the function $\fA$ satisfies the saddle point condition in a small game, i.e.,
        \begin{equation*}
            \begin{array}{c}
                \displaystyle
                \min_{p \in \USet} \max_{q \in \VSet} \langle s, \fA(t, r(\cdot), p, q) \rangle
                = \max_{q \in \VSet} \min_{p \in \USet} \langle s, \fA(t, r(\cdot), p, q) \rangle
            \end{array}
        \end{equation*}
        for any $(t, r(\cdot)) \in G^0$ and $s \in \mathbb{R}^n.$
\end{description}

According to (\ref{initial_condition_neutral_type}), if an initial position $(t_\ast, w_\ast(\cdot)) \in \G$ of original system (\ref{system}) is given, we define the corresponding initial position $(t_\ast, r_\ast(\cdot)) \in G^0$ of approximating system (\ref{system_y}) as follows:
\begin{equation} \label{r_ast}
    r_\ast (t) = \big(I^{1 - \alpha} (w_\ast(\cdot) - w_\ast(t_0)) \big) (t), \quad t \in [t_0, t_\ast].
\end{equation}
Due to Proposition~\ref{prop_G_properties} and the results given in Sect.~\ref{sec_Notations}, the function $r_\ast(\cdot)$ satisfies the inclusion $r_\ast(\cdot) \in \Lip_{M_1}^0([t_0, t_\ast], \mathbb{R}^n).$
Taking this into account, we call a position $(t, r(\cdot)) \in G^0$ of approximating system (\ref{system_y}) admissible if
\begin{equation*}
    \begin{array}{l}
        r(\cdot) \in \Lip^0([t_0, t], \mathbb{R}^n), \\[0.5em]
        \displaystyle
        \|\dot{r}(\tau)\| \leq (1 + \max_{\xi \in [t_0, \tau]} \|r(\xi)\|) \widetilde{c}_h \text{ for a.e. } \tau \in [t_0, t],
    \end{array}
\end{equation*}
where $\widetilde{c}_h = \max \{M_1, c_h\},$ and $c_h$ is the constant from condition ($B.3$).
The set of such admissible positions is denoted by $\GA.$
Note that this set is independent on the parameter $w_0.$

Let $(t_\ast, r_\ast(\cdot)) \in \GA$ and $t^\ast \in [t_\ast, \vartheta].$ As in Sect.~\ref{subsec_Positions}, by admissible control realizations of the players in the approximating game (\ref{system_y}), (\ref{quality_index_y}), we mean functions $p(\cdot) \in \mathcal{U}(t_\ast, t^\ast)$ and $q(\cdot) \in \mathcal{V}(t_\ast, t^\ast).$
Due to properties ($B.1$)--($B.3$), from the initial position $(t_\ast, r_\ast(\cdot)),$ such control realizations $p(\cdot)$ and $q(\cdot)$ uniquely generate the motion of approximating system (\ref{system_y}) that is the function $y(\cdot) \in \Lip^0([t_0, t^\ast], \mathbb{R}^n)$ satisfying the initial condition $y(t) = r_\ast(t),$ $t \in [t_0, t_\ast],$ and, together with $p(\cdot)$ and $q(\cdot),$ satisfying Eq. (\ref{system_y}) for almost every $t \in [t_\ast, t^\ast].$

Let us note the following properties of the set $\GA.$
Firstly, for any $(t_\ast, w_\ast(\cdot)) \in \G,$ the inclusion $(t_\ast, r_\ast(\cdot)) \in \GA$ is valid for the function $r_\ast(\cdot)$ defined by (\ref{r_ast}).
Secondly, for the motion $y(\cdot)$ of approximating system (\ref{system_y}) generated from $(t_\ast, r_\ast(\cdot)) \in \GA$ by $p(\cdot) \in \mathcal{U}(t_\ast, \vartheta)$ and $q(\cdot) \in \mathcal{V}(t_\ast, \vartheta),$ the inclusion $(t, y_t(\cdot)) \in \GA$ holds for any $t \in [t_\ast, \vartheta],$ where, according to (\ref{x_t}), we denote $y_t(\tau) = y(\tau),$ $\tau \in [t_0, t].$
Finally, the set $\GA$ is a compact subset of $G^0.$

Following the the scheme from \cite[Lemma~2]{Gomoyunov_2018} and taking into account that the constant $R_1$ in Proposition~\ref{prop_G_properties} does not depend on an initial position $(t_\ast, w_\ast(\cdot)) \in \G,$ one can prove the result below.
\begin{proposition} \label{prop_L}
    There exists $L_1 > 0$ such that the following statement holds.
    Let $(t_\ast, w_\ast(\cdot)) \in \G$ be an initial position of original system $(\ref{system}).$
    Let us consider approximating system $(\ref{system_y})$ for $w_0 = w_\ast(t_0),$ any $h > 0,$ and under the initial position $(t_\ast, r_\ast(\cdot))$ defined by $(\ref{r_ast}).$
    Then the inclusion $y(\cdot) \in \Lip_{L_1}^0([t_0, \vartheta], \mathbb{R}^n)$ is valid for any motion $y(\cdot)$ of the approximating system generated from $(t_\ast, r_\ast(\cdot))$ by $p(\cdot) \in \mathcal{U}(t_\ast, \vartheta)$ and $q(\cdot) \in \mathcal{V}(t_\ast, \vartheta).$
\end{proposition}

\subsection{The Value of the Approximating Game}
\label{subsec_AGV}

Let $w_0 \in B(R_0)$ and $h > 0$ be fixed.
Similarly to Sect.~\ref{subsec_NAS}, in the approximating differential game (\ref{system_y}), (\ref{quality_index_y}), one can consider non-anticipative strategies of the players and introduce the lower and upper game values, denoted respectively by $\valAQSF(t_\ast, r_\ast(\cdot))$ and $\valAQSS(t_\ast, r_\ast(\cdot)),$ $(t_\ast, r_\ast(\cdot)) \in \GA.$
From the results of \cite{Lukoyanov_2000,Lukoyanov_2003,Lukoyanov_2011} (see also \cite{Gomoyunov_Lukoyanov_Plaksin_2017}) it follows that, under conditions $(B.1)$--$(B.4),$ the approximating game has the value
\begin{equation*}
    \valA(t_\ast, r_\ast(\cdot)) = \valAQSF(t_\ast, r_\ast(\cdot)) = \valAQSS(t_\ast, r_\ast(\cdot)), \quad (t_\ast, r_\ast(\cdot)) \in \GA,
\end{equation*}
and, furthermore, this value can be guaranteed by the players if they use the positional strategies, described in the next section.

\subsection{Optimal Positional Strategies}

Let $w_0 \in B(R_0)$ and $h > 0$ be fixed.
In the approximating differential game (\ref{system}), (\ref{quality_index}), by the positional strategies $\SP_{w_0, h}$ and $\SQ_{w_0, h}$ of the players, we mean arbitrary functions
\begin{equation*}
    \begin{array}{l}
        \GA \times (0, \infty) \ni (t, r(\cdot), \varepsilon) \mapsto \SP_{w_0, h}(t, r(\cdot), \varepsilon) \in \USet, \\[0.5em]
        \GA \times (0, \infty) \ni (t, r(\cdot), \varepsilon) \mapsto \SQ_{w_0, h}(t, r(\cdot), \varepsilon) \in \VSet,
    \end{array}
\end{equation*}
where $\varepsilon$ is the accuracy parameter.

Let $(t_\ast, r_\ast(\cdot)) \in \GA,$ $\varepsilon > 0,$ and let
\begin{equation} \label{Delta}
    \Delta = \{\tau_j\}_{j = 1}^{k + 1},
    \quad \tau_1 = t_\ast, \quad \tau_{j+1} > \tau_j, \quad j \in \overline{1, k}, \quad \tau_{k+1} = \vartheta, \quad k \in \mathbb{N},
\end{equation}
be a partition of the interval $[t_\ast, \vartheta].$
The triple $\{\SP_{w_0, h}, \varepsilon, \Delta\}$ is called a control law of the first player.
This law forms in the approximating system a piecewise constant control realization $p(\cdot) \in \mathcal{U}(t_\ast, \vartheta)$ by the following step-by-step feedback rule:
\begin{equation} \label{control_law_1_ADG}
    p(t) = \SP_{w_0, h}(\tau_j, y_{\tau_j}(\cdot), \varepsilon),
    \quad t \in [\tau_j, \tau_{j + 1}), \quad j \in \overline{1, k},
\end{equation}
where $y_{\tau_1}(\cdot) = r_\ast(\cdot).$
Thus, from the initial position $(t_\ast, r_\ast(\cdot)),$ the control law of the first player $\{\SP_{w_0, h}, \varepsilon, \Delta\}$ together with a control realization of the second player $q(\cdot) \in \mathcal{V}(t_\ast, \vartheta)$ uniquely generate the motion $y(\cdot)$ of the approximating system and, therefore, determine the value $\gammaA$ of approximating quality index (\ref{quality_index_y}).

Similarly, we consider the control law of the second player $\{\SQ_{w_0, h}, \varepsilon, \Delta\},$ which forms a piecewise constant control realization $q(\cdot) \in \mathcal{V}(t_\ast, \vartheta)$ as follows:
\begin{equation} \label{control_law_2_ADG}
    q(t) = \SQ_{w_0, h}(\tau_j, y_{\tau_j}(\cdot), \varepsilon),
    \quad t \in [\tau_j, \tau_{j + 1}), \quad j \in \overline{1, k}.
\end{equation}
From the initial position $(t_\ast, r_\ast(\cdot)),$ the control law $\{\SQ_{w_0, h}, \varepsilon, \Delta\}$ together with $p(\cdot) \in \mathcal{U}(t_\ast, \vartheta)$ uniquely generate the motion $y(\cdot)$ of the approximating system and determine the value $\gammaA$ of approximating quality index (\ref{quality_index_y}).

By the scheme from \cite[Theorem~1]{Lukoyanov_2000} (see also \cite[Theorem~17.1]{Lukoyanov_2011}), one can prove the following lemma (see \cite{Gomoyunov_Lukoyanov_Plaksin_2017} for a related technique).
\begin{lemma} \label{lem_optimal_strategies}
    For any $w_0 \in B(R_0)$ and any $h > 0,$ in the approximating differential game $(\ref{system_y}),$ $(\ref{quality_index_y}),$ there exist the players' optimal positional strategies $\SPO$ and $\SQO$ that are optimal uniformly in $(t_\ast, r_\ast(\cdot)) \in \GA$ and $w_0 \in B(R_0).$
    Namely, for any $h > 0$ and any $\zeta > 0,$ one can choose $\varepsilon^{(1)} = \varepsilon^{(1)}(h, \zeta) > 0$ and $\delta^{(1)}(\varepsilon) = \delta^{(1)}(\varepsilon, h, \zeta) > 0,$ $\varepsilon \in (0, \varepsilon^{(1)}],$ such that the following statement holds.
    Let $w_0 \in B(R_0),$ $(t_\ast, r_\ast(\cdot)) \in \GA,$ $\varepsilon \in (0, \varepsilon^{(1)}],$ and let $\Delta$ be a partition $(\ref{Delta})$ with the diameter $\diam(\Delta) = \max_{j \in \overline{1, k}} (\tau_{j+1} - \tau_j) \leq \delta^{(1)}(\varepsilon).$
    Then the control law $\{\SPO, \varepsilon, \Delta\}$ of the first player guarantees for the value $\gammaA$ of approximating quality index $(\ref{quality_index_y})$ the inequality
    \begin{equation*}
        \gammaA \leq \valA(t_\ast, r_\ast(\cdot)) + \zeta
    \end{equation*}
    for any control realization of the second player $q(\cdot) \in \mathcal{V}(t_\ast, \vartheta);$
    and the control law $\{\SQO, \varepsilon, \Delta\}$ of the second player guarantees for the value $\gammaA$ of approximating quality index $(\ref{quality_index_y})$ the inequality
    \begin{equation*}
        \gammaA \geq \valA(t_\ast, r_\ast(\cdot)) - \zeta
    \end{equation*}
    for any control realization of the first player $p(\cdot) \in \mathcal{U}(t_\ast, \vartheta).$
\end{lemma}

Note that the uniformness in the parameter $w_0 \in B(R_0)$ is provided by the corresponding uniformness in properties ($B.1$)--($B.3$).

Let us describe shortly one of the ways of constructing such optimal strategies $\SPO$ and $\SQO.$
We apply the method of extremal shift to accompanying points (see, e.g., \cite{Krasovskii_Krasovskii_1995,Krasovskii_1985} and also \cite{Lukoyanov_2000,Lukoyanov_2011}).
For simplicity of the notation below, it is convenient to consider the so-called pre-strategies of the players in the approximating game (\ref{system_y}), (\ref{quality_index_y}).
Namely, by pre-strategies $\PSP_{w_0, h}$ and $\PSQ_{w_0, h}$ of the first and second players, we mean functions
\begin{equation*}
    \begin{array}{l}
        \GA \times \mathbb{R}^n \ni (t, r(\cdot), s) \mapsto \PSP_{w_0, h}(t, r(\cdot), s) \in \USet, \\[0.5em]
        \GA \times \mathbb{R}^n \ni (t, r(\cdot), s) \mapsto \PSQ_{w_0, h}(t, r(\cdot), s) \in \VSet
    \end{array}
\end{equation*}
that, for any $(t, r(\cdot)) \in \GA$ and any $s \in \mathbb{R}^n,$ satisfy the inclusions
\begin{equation*}
    \begin{array}{l}
        \displaystyle \PSP_{w_0, h}(t, r(\cdot), s) \in \argmin{p \in \USet} \max_{q \in \VSet}
        \langle s, \fA(t, r(\cdot), p, q) \rangle, \\[0.5em]
        \displaystyle \PSQ_{w_0, h}(t, r(\cdot), s) \in \argmax{q \in \VSet} \min_{p \in \USet}
        \langle s, \fA(t, r(\cdot), p, q\big) \rangle.
    \end{array}
\end{equation*}
Let $(t, r(\cdot)) \in \GA$ and $\varepsilon > 0.$ For the first and second players, we choose the accompanying points $r^{(p)}_\varepsilon(\cdot)$ and $r^{(q)}_\varepsilon(\cdot)$ from the conditions
\begin{equation*}
    r^{(p)}_\varepsilon(\cdot) \in \argmin{} \valA(t, r_\varepsilon(\cdot)), \quad
    r^{(q)}_\varepsilon(\cdot) \in \argmax{} \valA(t, r_\varepsilon(\cdot)),
\end{equation*}
where the minimum and maximum are calculated over the functions $r_\varepsilon(\cdot)$ such that
\begin{equation*}
    (t, r_\varepsilon(\cdot)) \in \GA, \quad
    \max_{\tau \in [t_0, t]} \|r(\tau) - r_\varepsilon(\tau)\| \leq (t - t_0) e^{(t - t_0) \lambda_h} \varepsilon,
\end{equation*}
and the constant $\lambda_h$ is chosen by the set $G_h^0$ in accordance with property ($B.2$).
Note that the minimum and maximum are attained due to continuity of the value function $\GA \ni (t_\ast, r_\ast(\cdot)) \mapsto \valA(t_\ast, r_\ast(\cdot)) \in \mathbb{R}.$ 
After that, we define
\begin{equation*}
    \begin{array}{l}
        \SPO(t, r(\cdot), \varepsilon) = \PSP_{w_0, h}(t, r(\cdot), r(t) - r^{(p)}_\varepsilon(t)), \\[0.5em]
        \SQO(t, r(\cdot), \varepsilon) = \PSQ_{w_0, h}(t, r(\cdot), r^{(q)}_\varepsilon(t) - r(t)).
    \end{array}
\end{equation*}

\begin{remark} \label{rem_methods}
    There are another methods for constructing the optimal positional strategies $\SPO$ and $\SQO$ (see, e.g., \cite{Krasovskii_Subbotin_1988,Lukoyanov_2000,Lukoyanov_2003,Lukoyanov_2011}).
    For example, if the value function $\valA$ is coinvariantly smooth, then the method of extremal shift in the direction of the coinvariant gradient of $\valA$ can be applied.
    In the general non-smooth case, such strategies can be constructed by the extremal shift in direction of the coinvariant gradient of a suitable coinvariantly smooth auxiliary function.
    Also, one can use the methods based on the notions of maximal $u$- and $v$-stable bridges.
    Furthermore, there are some specific methods for constructing the optimal strategies in the linear case (see, e.g., \cite{Gomoyunov_Lukoyanov_2012,Lukoyanov_Reshetova_1998}).
\end{remark}

\section{Players' Control Procedures with a Guide}
\label{sec_CPG}

In this section, we propose the players' feedback control procedures that use the optimally controlled approximating system (\ref{system_y}) as a guide.
It allows us to show that the values of the approximating differential games (\ref{system_y}), (\ref{quality_index_y}) have the limit when $h \downarrow 0.$
This fact constitutes the basis of the proof of the main result of the paper formulated in Theorem~\ref{thm_Existence} (see Sect.~\ref{sec_Proof}).

\subsection{Mutual Aiming Procedures between the Systems}
\label{subsec_Pcocedure}

According to \cite[Sect.~7]{Gomoyunov_2018}, let us consider the following mutual aiming procedure between original (\ref{system}) and approximating (\ref{system_y}) systems.
First of all, let us introduce pre-strategies of the players in the original game (\ref{system}), (\ref{quality_index}).
By pre-strategies $\PSU$ and $\PSV$ of the first and second players, we mean functions
\begin{equation*}
    \begin{array}{l}
        [t_0, \vartheta] \times \mathbb{R}^n \times \mathbb{R}^n \ni (t, x, s) \mapsto \PSU(t, x, s) \in \USet, \\[0.5em]
        [t_0, \vartheta] \times \mathbb{R}^n \times \mathbb{R}^n \ni (t, x, s) \mapsto \PSV(t, x, s) \in \VSet
    \end{array}
\end{equation*}
that, for any $t \in [t_0, \vartheta]$ and any $x, s \in \mathbb{R}^n,$ satisfy the inclusions
\begin{equation*}
    \begin{array}{l}
        \displaystyle \PSU(t, x, s) \in \argmin{u \in \USet} \max_{v \in \VSet} \langle s, f(t, x, u, v) \rangle, \\[0.5em]
        \displaystyle \PSV(t, x, s) \in \argmax{v \in \VSet} \min_{u \in \USet} \langle s, f(t, x, u, v) \rangle.
    \end{array}
\end{equation*}
Further, for $(t, w(\cdot)) \in \G$ and $(t, r(\cdot)) \in \GA,$ let us denote
\begin{equation*}
    s(t, w(\cdot), r(\cdot)) = w(t) - w(t_0) - h^{\alpha - 1} (\Delta_h^{1 - \alpha} r) (t).
\end{equation*}

Let $(t_\ast, w_\ast(\cdot)) \in \G$ be an initial position of original system (\ref{system}).
Let us fix $h > 0,$ put $w_0 = w_\ast(t_0),$ and consider the corresponding approximating system (\ref{system_y}) under the initial position $(t_\ast, r_\ast(\cdot))$ defined by (\ref{r_ast}).
Let us fix also a partition $\Delta$ (\ref{Delta}).
Let a first player's control realization $u(\cdot) \in \mathcal{U}(t_\ast, \vartheta)$ in the original system and a second player's control realization $q(\cdot) \in \mathcal{V}(t_\ast, \vartheta)$ in the approximating system be formed simultaneously according to the following step-by-step feedback rule:
\begin{equation} \label{procedure_1}
    u(t) =  \PSU(\tau_j, x(\tau_j), s_j), \quad q(t) = \PSQ_{w_0, h}(\tau_j, y_{\tau_j}(\cdot), s_j),  \quad t \in [\tau_j, \tau_{j+1}),
\end{equation}
where
\begin{equation} \label{s_j_1}
    s_j = s(\tau_j, x_{\tau_j}(\cdot), y_{\tau_j}(\cdot)),
\end{equation}
and $\PSQ_{w_0, h}$ is a pre-strategy of the second player in the approximating game.

\begin{lemma} \label{lem_procedure_1}
    For any $\xi > 0,$ there exist $h^{(2)} = h^{(2)}(\xi) > 0$ and $\delta^{(2)} = \delta^{(2)}(\xi) >0$ such that, for any initial position $(t_\ast, w_\ast(\cdot)) \in \G$ of original system $(\ref{system})$ and any partition $\Delta$ $(\ref{Delta})$ with the diameter $\diam(\Delta) \leq \delta^{(2)},$ the following statement is valid.
    Let us consider approximating system $(\ref{system_y})$ for $w_0 = w_\ast(t_0)$ and $h \in (0, h^{(2)}]$ under the initial position $(t_\ast, r_\ast(\cdot))$ defined by $(\ref{r_ast}).$
    Then, for any control realizations $v(\cdot) \in \mathcal{V}(t_\ast, \vartheta)$ and $p(\cdot) \in \mathcal{U}(t_\ast, \vartheta),$ if control realizations $u(\cdot) \in \mathcal{U}(t_\ast, \vartheta)$ and $q(\cdot) \in \mathcal{V}(t_\ast, \vartheta)$ are formed according to the mutual aiming procedure $(\ref{procedure_1}),$ $(\ref{s_j_1}),$ then the corresponding motions $x(\cdot)$ and $y(\cdot)$ of the original and approximating systems satisfy the inequality
    \begin{equation} \label{lem_procedure_1_main}
        \|x(t) - w_0 - h^{\alpha - 1} (\Delta_h^{1 - \alpha} y) (t)\| \leq \xi, \quad t \in [t_0, \vartheta].
    \end{equation}
\end{lemma}

The lemma is proved by the scheme from \cite[Theorem~3]{Gomoyunov_2018}, if we take into account that the constants $R_1$ and $H_1$ in Corollary~\ref{cor_motion_properties} and the constant $L_1$ in Proposition~\ref{prop_L} do not depend on an initial position $(t_\ast, w_\ast(\cdot)) \in \G.$

Similarly, one can consider another mutual aiming procedure between the original and approximating systems.
Namely, let $v(\cdot) \in \mathcal{V}(t_\ast, \vartheta)$ and $p(\cdot) \in \mathcal{U}(t_\ast, \vartheta)$ be formed on the basis of the partition $\Delta$ as follows:
\begin{equation} \label{procedure_2}
    v(t) =  \PSV(\tau_j, x(\tau_j), s_j), \quad p(t) = \PSP_{w_0, h}(\tau_j, y_{\tau_j}(\cdot), s_j),  \quad t \in [\tau_j, \tau_{j+1}),
\end{equation}
where
\begin{equation} \label{s_j_2}
    s_j = - s(\tau_j, x_{\tau_j}(\cdot), y_{\tau_j}(\cdot)).
\end{equation}
By analogy with Lemma~\ref{lem_procedure_1}, we obtain the following result.

\begin{lemma} \label{lem_procedure_2}
    For any $\xi > 0,$ there exist $h^{(3)} = h^{(3)}(\xi) > 0$ and $\delta^{(3)} = \delta^{(3)}(\xi) >0$ such that, for any initial position $(t_\ast, w_\ast(\cdot)) \in \G$ of original system $(\ref{system})$ and any partition $\Delta$ $(\ref{Delta})$ with the diameter $\diam(\Delta) \leq \delta^{(3)},$ the following statement is valid.
    Let us consider approximating system $(\ref{system_y})$ for $w_0 = w_\ast(t_0)$ and $h \in (0, h^{(3)}]$ under the initial position $(t_\ast, r_\ast(\cdot))$ defined by $(\ref{r_ast}).$
    Then, for any realizations $u(\cdot) \in \mathcal{U}(t_\ast, \vartheta)$ and $q(\cdot) \in \mathcal{V}(t_\ast, \vartheta),$ if realizations $v(\cdot) \in \mathcal{V}(t_\ast, \vartheta)$ and $p(\cdot) \in \mathcal{U}(t_\ast, \vartheta)$ are formed according to the mutual aiming procedure $(\ref{procedure_2}),$ $(\ref{s_j_2}),$ then the corresponding motions $x(\cdot)$ and $y(\cdot)$ of the original and approximating systems satisfy inequality $(\ref{lem_procedure_1_main}).$
\end{lemma}

\subsection{First Player's Control Procedure with a Guide}
\label{subsec_Control_procedure_1}

Let $(t_\ast, w_\ast(\cdot)) \in \G,$ $h > 0,$ $\varepsilon > 0,$ and a partition $\Delta$ (\ref{Delta}) be fixed.
We propose the following control procedure of the first player in the original differential game (\ref{system}), (\ref{quality_index}).
Let us consider the approximating differential game (\ref{system_y}), (\ref{quality_index_y}) for $w_0 = w_\ast(t_0),$ the fixed $h,$ and with the initial position $(t_\ast, r_\ast(\cdot))$ defined by (\ref{r_ast}).
By the steps of the partition $\Delta,$ the first player forms a control realization $u(\cdot) \in \mathcal{U}(t_\ast, \vartheta)$ in the original system and, at the same time, control realizations $p(\cdot) \in \mathcal{U}(t_\ast, \vartheta)$ and $q(\cdot) \in \mathcal{V}(t_\ast, \vartheta)$ in the approximating system as follows:
$u(\cdot)$ and $q(\cdot)$ are formed according to the mutual aiming procedure (\ref{procedure_1}), (\ref{s_j_1}), and $p(\cdot)$ is formed by the control law $\{\SPO, \varepsilon, \Delta\}$ (see (\ref{control_law_1_ADG})) on the basis of the optimal strategy $\SPO$ taken from Lemma~\ref{lem_optimal_strategies}.
Note that, from the initial position $(t_\ast, w_\ast(\cdot)),$ the described control procedure together with $v(\cdot) \in \mathcal{V}(t_\ast, \vartheta)$ uniquely generate the motion $x(\cdot)$ of the original system and, therefore, determine the value $\gamma$ of quality index (\ref{quality_index}).
Moreover, during this control procedure, the first player generates the auxiliary motion $y(\cdot)$ of the approximating system, which can be considered as a guide (see, e.g., \cite[\S~8.2]{Krasovskii_Subbotin_1988}).
For convenience, in what follows, the described control procedure is referred as $\UPROC(t_\ast, w_\ast(\cdot), h, \varepsilon, \Delta).$

For any $h > 0,$ let us introduce the function
\begin{equation} \label{valh}
    \valh_h(t_\ast, w_\ast(\cdot)) = \valA(t_\ast, r_\ast(\cdot)), \quad (t_\ast, w_\ast(\cdot)) \in \G,
\end{equation}
where $r_\ast(\cdot)$ is defined according to (\ref{r_ast}), and $\valA(t_\ast, r_\ast(\cdot))$ is the value of the approximating differential game (\ref{system_y}), (\ref{quality_index_y}) for $w_0 = w_\ast(t_0)$ and the fixed $h.$

\begin{lemma} \label{lem_U_h_varepsilon}
    For any $\zeta > 0,$ there exist
    \begin{equation*}
        \begin{array}{l}
            h^{(4)} = h^{(4)}(\zeta) > 0, \\[0.5em]
            \varepsilon^{(4)}(h) = \varepsilon^{(4)}(h, \zeta) > 0, \quad h \in (0, h^{(4)}], \\[0.5em]
            \delta^{(4)}(\varepsilon, h) = \delta^{(4)}(\varepsilon, h, \zeta) > 0,
            \quad \varepsilon \in (0, \varepsilon^{(4)}(h)], \quad h \in (0, h^{(4)}],
        \end{array}
    \end{equation*}
    such that, for any $(t_\ast, w_\ast(\cdot)) \in \G,$ $h \in (0, h^{(4)}],$ $\varepsilon \in (0, \varepsilon^{(4)}(h)],$ and any partition $\Delta$ $(\ref{Delta})$ with the diameter $\diam(\Delta) \leq \delta^{(4)}(\varepsilon, h),$ the first player's control procedure with a guide $\UPROC(t_\ast, w_\ast(\cdot), h, \varepsilon, \Delta)$ guarantees for the value $\gamma$ of quality index $(\ref{quality_index})$ the inequality
    \begin{equation*}
        \gamma \leq \valh_h(t_\ast, w_\ast(\cdot)) + \zeta
    \end{equation*}
    for any control realization of the second player $v(\cdot) \in \mathcal{V}(t_\ast, \vartheta).$
\end{lemma}
\begin{proof}
    Applying \cite[Proposition~7]{Gomoyunov_2018}, by the constant $L_1$ from Proposition~\ref{prop_L}, one can choose $R_2 > 0$ and $H_2 > 0$ such that the inequalities
    \begin{equation*}
        \begin{array}{l}
            \|h^{\alpha - 1} (\Delta_h^{1 - \alpha} y)(t)\| \leq R_2, 
            \\[0.5em]
            \|h^{\alpha - 1} (\Delta_h^{1 - \alpha} y)(t) - h^{\alpha - 1} (\Delta_h^{1 - \alpha} y)(t^\prime)\|
            \leq H_2 |t - t^\prime|^\alpha, \quad t, t^\prime \in [t_0, \vartheta],
        \end{array}
    \end{equation*}
    are valid for any $h > 0$ and any $y(\cdot) \in \Lip^0_{L_1}([t_0, \vartheta], \mathbb{R}^n).$
    Taking the constants $R_1$ and $H_1$ from Corollary~\ref{cor_motion_properties}, we define $R_3 = \max\{R_1, R_0 + R_2\},$ $H_3 = \max\{H_1, H_2\},$ and consider the compact set $D \subset \C([t_0, \vartheta], \mathbb{R}^n)$ consisting of the functions $x(\cdot)$ such that
    \begin{equation*}
        \|x(t)\| \leq R_3, \quad \|x(t) - x(t^\prime)\| \leq H_3 |t - t^\prime|^\alpha, \quad t, t^\prime \in [t_0, \vartheta].
    \end{equation*}
    Let $\zeta > 0$ be fixed.
    Due to ($A.5$), there exists $\xi = \xi(\zeta) > 0$ such that, for any $x(\cdot), x^\prime(\cdot) \in D,$ from the inequality $\|x(\cdot) - x^\prime(\cdot)\|_\infty \leq \xi$ it follows that $|\sigma(x(\cdot)) - \sigma(x^\prime(\cdot))| \leq \zeta / 2.$
    Let us choose $h^{(2)}(\xi) > 0$ and $\delta^{(2)}(\xi) > 0$ by Lemma~\ref{lem_procedure_1}, and put $h^{(4)} = h^{(2)}(\xi).$
    Finally, for any $h \in (0, h^{(4)}],$ we take $\varepsilon^{(1)} (h, \zeta/2) > 0$ and $\delta^{(1)}(\varepsilon, h, \zeta/2) > 0,$ $\varepsilon \in (0, \varepsilon^{(1)}(h, \zeta/2)],$ from Lemma~\ref{lem_optimal_strategies}, and define
    \begin{equation*}
        \begin{array}{l}
            \varepsilon^{(4)}(h) = \varepsilon^{(1)} (h, \zeta/2), \\[0.5em]
            \delta^{(4)}(\varepsilon, h) = \min \big\{\delta^{(1)}(\varepsilon, h, \zeta/2), \delta^{(2)}(\xi) \big\},
            \quad \varepsilon \in (0, \varepsilon^{(4)}(h)].
        \end{array}
    \end{equation*}
    Let us show that the statement of the lemma is valid for the chosen parameters.

    Let $(t_\ast, w_\ast(\cdot)) \in \G,$ $h \in (0, h^{(4)}],$ $\varepsilon \in (0, \varepsilon^{(4)}(h)],$ and let $\Delta$ be a partition (\ref{Delta}) with the diameter $\diam(\Delta) \leq \delta^{(4)}(\varepsilon, h).$
    Let us consider the motion $x(\cdot)$ of system (\ref{system}) generated from the initial position $(t_\ast, w_\ast(\cdot))$ by the first player's control procedure with a guide $\UPROC = \UPROC(t_\ast, w_\ast(\cdot), h, \varepsilon, \Delta)$ and a second player's control realization $v(\cdot) \in \mathcal{V}(t_\ast, \vartheta).$
    Let us consider the corresponding first player's control realization $u(\cdot) \in \mathcal{U}(t_\ast, \vartheta)$ in the original system and players' control realizations $p(\cdot) \in \mathcal{U}(t_\ast, \vartheta)$ and $q(\cdot) \in \mathcal{V}(t_\ast, \vartheta)$ in the approximating system (\ref{system_y}) for $w_0 = w_\ast(t_0),$ the fixed $h,$ and with the initial position $(t_\ast, r_\ast(\cdot))$ defined by (\ref{r_ast}).
    Let $y(\cdot)$ be the corresponding motion of the approximating system.
    By the definition of $U,$ the motion $y(\cdot)$ is generated by the control law $\{\SPO, \varepsilon, \Delta\}$ on the basis of the first player's optimal positional strategy $\SPO.$
    Hence, for the auxiliary function $x^\prime(t) = w_0 + h^{\alpha - 1} (\Delta^{1 - \alpha}_h y)(t),$ $t \in [t_0, \vartheta],$ due to the choice of $\varepsilon$ and $\Delta,$ we obtain
    \begin{equation*}
        \sigma(x^\prime(\cdot)) = \sigmaA(y(\cdot)) = \gammaA \leq \valA(t_\ast, r_\ast(\cdot)) + \zeta/2
        = \valh_h(t_\ast, w_\ast(\cdot)) + \zeta/2.
    \end{equation*}
    Moreover, the control realizations $u(\cdot)$ and $q(\cdot)$ are formed according to the mutual aiming procedure (\ref{procedure_1}), (\ref{s_j_1}).
    Therefore, according to the choice of $h$ and $\Delta,$ we derive $\|x(\cdot) - x^\prime(\cdot)\|_\infty \leq \xi.$
    Thus, taking into account the inclusions $x(\cdot), x^\prime(\cdot) \in D,$ by the choice of $\xi,$ we have
    \begin{equation*}
        \gamma = \sigma(x(\cdot)) \leq \sigma(x^\prime(\cdot)) + \zeta/2 \leq \valh_h(t_\ast, w_\ast(\cdot)) + \zeta.
    \end{equation*}
    The lemma is proved. \hfill $\square$
\end{proof}

\subsection{Second Player's Control Procedure with a Guide}
\label{subsec_Control_procedure_2}

Similarly to Sect.~\ref{subsec_Control_procedure_1}, we propose the following second player's control procedure with a guide in the original differential game (\ref{system}), (\ref{quality_index}).
Let $(t_\ast, w_\ast(\cdot)) \in \G,$ $h > 0,$ $\varepsilon > 0,$ and a partition $\Delta$ (\ref{Delta}) be fixed.
Let us consider the approximating differential game (\ref{system_y}), (\ref{quality_index_y}) for $w_0 = w_\ast(t_0),$ the fixed $h,$ and with the initial position $(t_\ast, r_\ast(\cdot))$ defined by (\ref{r_ast}).
By the steps of the partition $\Delta,$ the second player forms a control realization $v(\cdot) \in \mathcal{V}(t_\ast, \vartheta)$ in the original system and, at the same time, control realizations $p(\cdot) \in \mathcal{U}(t_\ast, \vartheta)$ and $q(\cdot) \in \mathcal{V}(t_\ast, \vartheta)$ in the approximating system as follows:
$v(\cdot)$ and $p(\cdot)$ are formed according to the mutual aiming procedure (\ref{procedure_2}), (\ref{s_j_2}), and $q(\cdot)$ is formed by the control law $\{\SQO, \varepsilon, \Delta\}$ (see (\ref{control_law_2_ADG})) on the basis of the optimal strategy $\SQO$ taken from Lemma~\ref{lem_optimal_strategies}.
From the initial position $(t_\ast, w_\ast(\cdot)),$ the described control procedure together with $u(\cdot) \in \mathcal{U}(t_\ast, \vartheta)$ uniquely generate the motion $x(\cdot)$ of the original system and determine the value $\gamma$ of quality index (\ref{quality_index}).
In what follows, this control procedure with a guide is referred as $\VPROC(t_\ast, w_\ast(\cdot), h, \varepsilon, \Delta).$

By analogy with Lemma~\ref{lem_U_h_varepsilon}, on the basis of Lemma~\ref{lem_procedure_2}, the following result can be proved.
\begin{lemma} \label{lem_V_h_varepsilon}
    For any $\zeta > 0,$ there exist
    \begin{equation*}
        \begin{array}{l}
            h^{(5)} = h^{(5)}(\zeta) > 0, \\[0.5em]
            \varepsilon^{(5)}(h) = \varepsilon^{(5)}(h, \zeta) > 0, \quad h \in (0, h^{(5)}], \\[0.5em]
            \delta^{(5)}(\varepsilon, h) = \delta^{(5)}(\varepsilon, h, \zeta) > 0,
            \quad \varepsilon \in (0, \varepsilon^{(5)}(h)], \quad h \in (0, h^{(5)}],
        \end{array}
    \end{equation*}
    such that, for any $(t_\ast, w_\ast(\cdot)) \in \G,$ $h \in (0, h^{(5)}],$ $\varepsilon \in (0, \varepsilon^{(5)}(h)],$ and any partition $\Delta$ $(\ref{Delta})$ with the diameter $\diam(\Delta) \leq \delta^{(5)}(\varepsilon, h),$ the second player's control procedure with a guide $\VPROC(t_\ast, w_\ast(\cdot), h, \varepsilon, \Delta)$ guarantees for the value $\gamma$ of quality index $(\ref{quality_index})$ the inequality
    \begin{equation*}
        \gamma \geq \valh_h(t_\ast, w_\ast(\cdot)) - \zeta
    \end{equation*}
    for any control realization of the first player $u(\cdot) \in \mathcal{U}(t_\ast, \vartheta).$
\end{lemma}

\subsection{Limit of the Values of the Approximating Games}
\label{subsec_Limit}

Considering in the original differential game (\ref{system}), (\ref{quality_index}) the case when the both players use the described in Sect.~\ref{subsec_Control_procedure_1} and~\ref{subsec_Control_procedure_2} control procedures with a guide, we obtain the result below.
\begin{lemma} \label{lem_limit}
    For any initial position $(t_\ast, w_\ast(\cdot)) \in \G,$ the following limit exists:
    \begin{equation} \label{lem_limit_main}
        \lim_{h \downarrow 0} \valh_h (t_\ast, w_\ast(\cdot)) = \vall (t_\ast, w_\ast(\cdot)),
    \end{equation}
    where $\valh_h (t_\ast, w_\ast(\cdot))$ is defined by $(\ref{valh}).$
    Moreover, the convergence is uniform in $(t_\ast, w_\ast(\cdot)) \in \G.$
\end{lemma}
\begin{proof}
    By the Cauchy criterion, to prove the lemma, it is sufficient to show that, for any $\zeta > 0,$ there exists $h = h(\zeta) > 0$ such that, for any $h_1, h_2 \in (0, h]$ and any $(t_\ast, w_\ast(\cdot)) \in \G,$ the inequality below is valid:
    \begin{equation} \label{lem_limit_proof}
        \valh_{h_2} (t_\ast, w_\ast(\cdot)) \leq \valh_{h_1} (t_\ast, w_\ast(\cdot)) + \zeta.
    \end{equation}
    Let $\zeta > 0$ be fixed. By Lemmas~\ref{lem_U_h_varepsilon} and~\ref{lem_V_h_varepsilon}, for $i \in \{4, 5\},$ let us choose
    \begin{equation} \label{h_i}
        \begin{array}{l}
            h^{(i)} = h^{(i)}(\zeta/2) > 0, \\[0.5em]
            \varepsilon^{(i)}(h) = \varepsilon^{(i)}(h, \zeta/2) > 0, \quad h \in (0, h^{(i)}], \\[0.5em]
            \delta^{(i)}(\varepsilon, h) = \delta^{(i)}(\varepsilon, h, \zeta/2) > 0,
            \quad \varepsilon \in (0, \varepsilon^{(i)}(h)], \quad  h \in (0, h^{(i)}],
        \end{array}
    \end{equation}
    and put $h = \min \{h^{(4)}, h^{(5)}\}.$
    Let $h_1, h_2 \in (0, h].$
    We define
    \begin{equation*}
        \varepsilon = \min \big\{\varepsilon^{(4)}(h_1), \, \varepsilon^{(5)}(h_2) \big\},
        \quad \delta = \min \big\{\delta^{(4)}(\varepsilon, h_1), \, \delta^{(5)}(\varepsilon, h_2)\big\}.
    \end{equation*}
    Let $(t_\ast, w_\ast(\cdot)) \in \G,$ and $\Delta$ be a partition (\ref{Delta}) with the diameter $\diam(\Delta) \leq \delta.$ Let us consider the motion $x(\cdot)$ of system (\ref{system}) generated by the players' control procedures with a guide $\UPROC(t_\ast, w_\ast(\cdot), h_1, \varepsilon, \Delta)$ and $\VPROC(t_\ast, w_\ast(\cdot), h_2, \varepsilon, \Delta).$
    Then, for the realized value $\gamma = \sigma(x(\cdot))$ of quality index (\ref{quality_index}), due to the choice of $h_1,$ $h_2,$ $\varepsilon$ and $\Delta,$ we have
    \begin{equation*}
        \valh_{h_2}(t_\ast, w_\ast(\cdot)) - \zeta/2 \leq \gamma \leq \valh_{h_1}(t_\ast, w_\ast(\cdot)) + \zeta/2,
    \end{equation*}
    wherefrom we derive (\ref{lem_limit_proof}).
    The lemma is proved. \hfill $\square$
\end{proof}

\section{Value of the Game}
\label{sec_Proof}

The main result of the paper is the following.
\begin{theorem} \label{thm_Existence}
    Let conditions $(A.1)$--$(A.5)$ be satisfied.
    Then:
    \begin{enumerate}
      \item The differential game $(\ref{system}),$ $(\ref{initial_condition})$ has the value $\Val(t_\ast, w_\ast(\cdot)),$ $(t_\ast, w_\ast(\cdot)) \in \G.$
      \item This value coincides with the limit $\vall(t_\ast, w_\ast(\cdot))$ $($see $(\ref{lem_limit_main}))$ of the values of the approximating differential games $(\ref{system_y}),$ $(\ref{quality_index_y}).$
      \item For any $\zeta > 0,$ there exist
        \begin{equation*}
            \begin{array}{l}
                h_\ast = h_\ast(\zeta) > 0, \\[0.5em]
                \varepsilon_\ast(h) = \varepsilon_\ast(h, \zeta) > 0, \quad h \in (0, h_\ast], \\[0.5em]
                \delta_\ast(\varepsilon, h) = \delta_\ast(\varepsilon, h, \zeta) > 0,
                \quad \varepsilon \in (0, \varepsilon_\ast(h)], \quad h \in (0, h_\ast],
            \end{array}
        \end{equation*}
        such that the following statement holds.
        Let $(t_\ast, w_\ast(\cdot)) \in \G,$ $h \in (0, h_\ast],$ $\varepsilon \in (0, \varepsilon_\ast(h)],$ and let $\Delta$ be a partition $(\ref{Delta})$ with the diameter $\diam(\Delta) \leq \delta_\ast(\varepsilon, h).$
        Then the control procedure with a guide of the first player $\UPROC(t_\ast, w_\ast(\cdot), h, \varepsilon, \Delta)$ guarantees for the value $\gamma$ of quality index $(\ref{quality_index})$ the inequality
        \begin{equation} \label{gamma_leq_Gamma_zeta}
            \gamma \leq \Val(t_\ast, w_\ast(\cdot)) + \zeta
        \end{equation}
        for any control realization of the second player $u(\cdot) \in \mathcal{V}(t_\ast, \vartheta);$
        and the control procedure with a guide of the second player $\VPROC(t_\ast, w_\ast(\cdot), h, \varepsilon, \Delta)$ guarantees for the value $\gamma$ of quality index $(\ref{quality_index})$ the inequality
        \begin{equation} \label{gamma_geq_Gamma_zeta}
            \gamma \geq \Val(t_\ast, w_\ast(\cdot)) - \zeta
        \end{equation}
        for any control realization of the first player $u(\cdot) \in \mathcal{U}(t_\ast, \vartheta).$
    \end{enumerate}
\end{theorem}
\begin{proof}
    Let $\zeta > 0$ be fixed.
    Let us define
    \begin{equation*}
        \begin{array}{l}
            h_\ast = \min \big\{h^{(4)}, h^{(5)}, h^{(6)}\big\}, \\[0.5em]
            \varepsilon_\ast(h) = \min \big\{\varepsilon^{(4)}(h), \varepsilon^{(5)}(h)\big\}, \quad h \in (0, h_\ast], \\[0.5em]
            \delta_\ast(\varepsilon, h) = \min \big\{\delta^{(4)}(\varepsilon, h), \delta^{(5)}(\varepsilon, h)\big\},
            \quad \varepsilon \in (0, \varepsilon_\ast], \quad h \in (0, h_\ast],
        \end{array}
    \end{equation*}
    where $h^{(i)} > 0,$ $\varepsilon^{(i)}(h) > 0$ and $\delta^{(i)}(\varepsilon, h) > 0$ for $i \in \{4, 5\}$ are chosen as in (\ref{h_i}), and $h^{(6)} = h^{(6)}(\zeta/2) > 0$ is chosen according to Lemma~\ref{lem_limit} such that, for any $h \in (0, h^{(6)}]$ and any $(t_\ast, w_\ast(\cdot)) \in \G,$ the inequality below is valid:
    \begin{equation*}
        |\vall(t_\ast, w_\ast(\cdot)) - \valh_h(t_\ast, w_\ast(\cdot))| \leq \zeta/2.
    \end{equation*}
    Let $(t_\ast, w_\ast(\cdot)) \in \G,$ $h \in (0, h_\ast],$ $\varepsilon \in (0, \varepsilon_\ast(h)],$ and let $\Delta$ be a partition $(\ref{Delta})$ with the diameter $\diam(\Delta) \leq \delta_\ast(\varepsilon, h).$
    Let us consider the first player's control procedure with a guide $\UPROC = \UPROC(t_\ast, w_\ast(\cdot), h, \varepsilon, \Delta).$
    On the basis of this procedure, we define the first player's non-anticipative strategy $\alpha$ (see Sect.~\ref{subsec_NAS}) as follows.
    For any $v(\cdot) \in \mathcal{V}(t_\ast, \vartheta),$ we consider the unique motion $x(\cdot)$ of system (\ref{system}) and the control realization $u(\cdot)$ that are formed by $\UPROC$ and $v(\cdot),$ and put $\alpha(v(\cdot)) = u(\cdot).$
    Further, since, by the choice of $h,$ $\varepsilon$ and $\Delta,$ for any $v(\cdot) \in \mathcal{V}(t_\ast, \vartheta),$ the corresponding value $\gamma = \sigma(x(\cdot))$ of quality index (\ref{quality_index}) satisfies the inequality
    \begin{equation} \label{thm_Existence_proof_1}
        \gamma \leq \valh_h(t_\ast, w_\ast(\cdot)) + \zeta/2 \leq \vall(t_\ast, w_\ast(\cdot)) + \zeta,
    \end{equation}
    then, by definition (\ref{lower_value}) of the lower game value $\valQSF(t_\ast, w_\ast(\cdot)),$ we obtain
    \begin{equation*}
        \valQSF(t_\ast, w_\ast(\cdot)) \leq \vall(t_\ast, w_\ast(\cdot)) + \zeta.
    \end{equation*}
    Taking into account that this inequality is valid for any $\zeta > 0,$ we derive
    \begin{equation*}
        \valQSF(t_\ast, w_\ast(\cdot)) \leq \vall(t_\ast, w_\ast(\cdot)).
    \end{equation*}

    Now, arguing by contradiction, let us suppose that
    \begin{equation} \label{thm_quasi-strategies_proof}
        \valQSF(t_\ast, w_\ast(\cdot)) + \zeta^\ast = \vall(t_\ast, w_\ast(\cdot))
    \end{equation}
    for a number $\zeta^\ast > 0.$
    Let $\alpha^\ast$ be a first player's non-anticipative strategy such that, for any $v(\cdot) \in \mathcal{V}(t_\ast, \vartheta),$ the motion $x(\cdot)$ of system (\ref{system}) generated from $(t_\ast, w_\ast(\cdot))$ by $v(\cdot)$ and $u(\cdot) = \alpha^\ast(v(\cdot))$ satisfies the inequality
    \begin{equation*}
        \gamma = \sigma(x(\cdot)) \leq \valQSF(t_\ast, w_\ast(\cdot)) + \zeta^\ast/3.
    \end{equation*}
    Similarly to above, based on Lemmas~\ref{lem_V_h_varepsilon} and~\ref{lem_limit}, by the number $\zeta^\ast/3,$ one can choose $h^\ast > 0,$ $\varepsilon^\ast > 0,$ and a partition $\Delta^\ast$ (\ref{Delta}) such that the motion $x(\cdot)$ of system (\ref{system}) generated from $(t_\ast, w_\ast(\cdot))$ by the second player's control procedure with a guide $\VPROC = \VPROC(t_\ast, w_\ast(\cdot), h^\ast, \varepsilon^\ast, \Delta^\ast)$ and a first player's control realization $u(\cdot) \in \mathcal{U}(t_\ast, \vartheta)$ satisfies the inequality
    \begin{equation*}
        \gamma = \sigma(x(\cdot)) \geq \vall(t_\ast, w_\ast(\cdot)) - \zeta^\ast/3.
    \end{equation*}
    According to the definition (see Sect~\ref{subsec_Control_procedure_2}), the control procedure $\VPROC$ forms $v(\cdot)$ by the steps of the partition $\Delta^\ast$ on the basis of the information about the realized values of the state vectors of the original and approximating systems.
    Therefore, since $\alpha^\ast$ is non-anticipative, one can consider the motion $x^\ast(\cdot)$ generated by $u(\cdot) \in \mathcal{U}(t_\ast, \vartheta)$ and $v(\cdot) \in \mathcal{V}(t_\ast, \vartheta)$ such that $u(\cdot) = \alpha^\ast(v(\cdot)),$ and, at the same time, $v(\cdot)$ is formed by $\VPROC.$
    For this motion $x^\ast(\cdot),$ we have
    \begin{equation*}
        \vall(t_\ast, w_\ast(\cdot)) - \zeta^\ast/3 \leq \sigma(x^\ast(\cdot)) \leq \valQSF(t_\ast, w_\ast(\cdot)) + \zeta^\ast/3,
    \end{equation*}
    wherefrom we obtain
    \begin{equation*}
        \vall(t_\ast, w_\ast(\cdot)) \leq \valQSF(t_\ast, w_\ast(\cdot)) + 2 \zeta^\ast/3.
    \end{equation*}
    The obtained inequality contradicts (\ref{thm_quasi-strategies_proof}) since $\zeta^\ast > 0.$
    Hence, we derive
    \begin{equation} \label{thm_Existence_proof_2}
        \valQSF(t_\ast, w_\ast(\cdot)) = \vall(t_\ast, w_\ast(\cdot)).
    \end{equation}
    The validity of the equality $\valQSS(t_\ast, w_\ast(\cdot)) = \vall(t_\ast, w_\ast(\cdot))$ can be established in a similar way.
    Thus, the first and second parts of the theorem are proved.
    Inequality (\ref{gamma_leq_Gamma_zeta}) in the third part of the theorem follows directly from (\ref{thm_Existence_proof_1}) and (\ref{thm_Existence_proof_2}).
    The validity of inequality (\ref{gamma_geq_Gamma_zeta}) can be shown similarly.
    The theorem is proved. \hfill $\square$
\end{proof}

\begin{remark}
    Let us note that, following \cite[\S~8.2]{Krasovskii_Subbotin_1988} (see also \cite{Lukoyanov_Plaksin_2015} for details), one can consider another formalization of the differential game (\ref{system}), (\ref{quality_index}).
    Namely, one can formally describe a sufficiently wide classes of players' strategies with a guide and introduce the corresponding values of the players' optimal guaranteed results.
    One can show that from Theorem~\ref{thm_Existence} it follows that these optimal guaranteed results coincide, i.e., the differential game has the value in the classes of strategies with a guide, and this value is equal to $\Val(t_\ast,w_\ast(\cdot)).$
    Moreover, the players' strategies with a guide that guarantee inequalities (\ref{gamma_leq_Gamma_zeta}) and (\ref{gamma_geq_Gamma_zeta}) can be constructed on the basis of the proposed in Sects.~\ref{subsec_Control_procedure_1} and~\ref{subsec_Control_procedure_2} control procedures.
    In this sense, these control procedures with a guide can be called optimal.
\end{remark}

\begin{remark}
    In addition to Remark~\ref{rem_methods}, another possible way of solving the approximating differential game (\ref{system_y}), (\ref{quality_index_y}) is to approximate functional differential equation of a retarded type (\ref{system_y}) by a high-dimensional system of ordinary differential equations (see, e.g., \cite{Lukoyanov_Plaksin_2015_2} and the references therein).
    Note that this approach can also be used for proving the existence of the game value and constructing the players' optimal control procedures with a guide in the original differential game (\ref{system}), (\ref{quality_index}).
\end{remark}

\section{Conclusion}
\label{sec_conclusion}

In the paper, we have considered a zero-sum differential game in a dynamical system which motion is described by a fractional differential equation.
We have proved that the lower and upper game values coincide, i.e., the differential game has the value.
The proof is based on the appropriate approximation of the game by a differential game in a dynamical system which motion is described by a first order functional differential equation of a retarded type.
This approach has also allowed us to propose the optimal players' feedback control procedures with a guide, which can be effectively applied if the optimal in the approximating game players' positional strategies are found.

\end{document}